\documentclass{amsart}
\usepackage{amsmath,amstext,amsbsy,amssymb}

\newtheorem{theorem}{Theorem}

\newtheorem{proposition}[theorem]{Proposition}
\newtheorem{lemma}[theorem]{Lemma}

\begin{document}

\title{Realizations of Affine Lie Algebra ${A^{(1)}_1}$ at Negative Levels}
\author{Jilan Dong}
\address{School of Sciences, South China University of Technology, Guangzhou, Guangdong 510640, China} \email{dongjilan.11@163.com}
\author{Naihuan Jing$^*$}
\address{Department of Mathematics, North Carolina State University, Raleigh, NC 27695, USA}
\email{jing@math.ncsu.edu}

\thanks{{\scriptsize
\hskip -0.4 true cm MSC (2010): Primary: 17B30; Secondary: 17B68.
$*$Corresponding author.
}}

\maketitle

\begin{abstract} A realization of the affine Lie
algebra ${A^{(1)}_1}$ and the relevant $Z$-algebra
at negative level $-k$ is given in terms of parafermions. This generalizes the
recent work on realization of the affine Lie algebra at the critical level.
\end{abstract}

\maketitle


\section{Introduction}
\label{sec:1} Since the sixties the theory of
affine Lie algebras has been one of the popular subjects in
mathematical physics. Its physical applications and mathematical
properties usually depend on whether one can use the matrix method
to give a concrete realization or representation. This approach has
been used in dual resonance models, infinitesimal B\"{a}cklund
transformations in soliton theory etc. The first concrete
realization of the affine Lie algebra $\widehat{sl}(2)$ was
Lepowsky-Wilson's vertex operator representation at level one
\cite{s1} and was then generalized to arbitrary types by
Kac-Kazhdan-Lepowsky-Wilson \cite{KKLW}. Later the homogeneous
realization of simply laced affine Lie algebras at level one was
given by I. Frenkel-Kac \cite{FK} and Segal \cite{S}. Fermionic
realizations were also constructed by I. Frenkel \cite{F} and
Kac-Peterson \cite{KP}, and were generalized to arbitrary types by
Feingold and I. Frenkel \cite{FF}.

Representations of the affine Lie algebras at other levels also have
attracted a lot of attention \cite{s4,s5,s6,s7,s8}. Wakimoto
\cite{s2} derived a general scheme to realize the affine Lie algebra
of type $A_1^{(1)}$ and this was generalized to higher rank by
Feigin and E. Frenkel \cite{FeF}. Generally speaking, highest weight
representations of the affine Lie algebras with integral levels are
built from the theory of vertex operators in terms of bosonic or
fermionic operators. Recently Adamovic \cite{s3} used vertex superalgebras
to study critical modules for $\widehat{sl}_2(\mathbb{C})$. 
Dunbar et al \cite{s6} also gave
a new representation of $\widehat{sl}_2(\mathbb{C})$ at the
critical level using some technique similar to semi-infinite wedge
products. This paper is a generalization of their work to arbitrary
negative integral level using the theory of parafermions.
Parafermions are introduced in statistical mechanics and conformal
field theory, they are also related to Majorana fermions, fractional
superstring, mirror symmetry, and have close connections with
exclusion statistics, quantum computations, Bose-Einstein
condensates etc \cite{s9,s10,s11}. In particular, the parafermions,
sometimes regarded as $Z$-algebras proposed in \cite{s12} contribute
to various extensions of the Ising model and 3-state Potts model,
all of which are basically relevant to ${A^{(1)}_1}$. These works
show that the $Z$-algebra at a positive integral level is identical
with that of ${A^{(1)}_1}$-parafermions.

In the monograph \cite{s7}, Dong and Lepowsky constructed canonical generalized vertex operator algebras for $\widehat{\mathfrak{g}}$ (of simply laced types $\hat{A}$, $\hat{D}$ or $\hat{E}$) and pointed out that the corresponding quotient space for the vacuum space of any positive integer level k standard $\hat{\mathfrak{G}}$-module is a module of the generalized vertex operator algebra. Furthermore, as an illustration, they showed in details the construction for ${A^{(1)}_1}$.  They used the vacuum space of $L(k,0)$ ($k\in\mathbb N$) in terms of a natural Heisenberg subalgebra of ${A^{(1)}_1}$ to define a quotient spaces of this vacuum space by the action of an infinite cyclic group, and then realized the parafermion algebra as the canonically modified $Z$-algebra acting on certain quotient spaces.

In the recent work of \cite{s6} the affine Lie algebra
$\widehat{sl}_2(\mathbb{C})$ at the critical level $-2$ was realized
using the generalized Clifford algebra. This shows that the case of
negative level can be treated by parafermions as well. In this
paper we generalize this result and realize the affine Lie algebra $\widehat{sl}_2$
 at negative levels by parafermions. Although many results at negative
levels are quite similar to the positive integral levels, we still
give a complete treatment of the realization with the hope that this
may be useful to understand Lusztig's theory of the relationship
between quantum groups and affine Lie algebras. For completeness we include all necessary computation of
operator product expansions of parafermions and also provide the detailed
verifications of the $Z$-algebra relations.

 The paper is organized as follows. In section two we first recall the basic definitions. The later part of section two reviews some basic results of parafermions based on \cite{s6}, and briefly explains the physicists' approach to parafermion fields with respect to each form of current algebras, operator product expansions and so on, and also studies in detail the generalized commutation relations and in particular modifications needed in the paper. In section three, a parafermionic representation of ${A^{(1)}_1}$ at level $-k$ ($k\in\mathbb N$) is constructed and corresponding results for the associated $Z$-algebra are given.

\section{Basic definitions}
\label{sec:2}
\subsection{The affine Lie algebra $\widehat{sl}_2$}
Let $\widehat{sl}_2(\mathbb{C})$ be the affine Lie algebra of type ${A^{(1)}_1}$, which is generated by a 1-dimensional central $\emph{c}$, a degree derivation $d=1\otimes t\partial t$ and elements $a(m)=a\otimes t^m\in sl_2(\mathbb{C})\otimes \mathbb{C} [t,t^{-1}]$, where $\mathbb{C} [t,t^{-1}]$ is the algebra of Laurent polynomials in the indeterminate $t$. The Lie bracket operation is defined by

\begin{equation}\label{eq:2.1}
\begin{array}{cc} &[c,\widehat{sl}_2(\mathbb{C})],\qquad [d,a(m)]=ma(m)\\&[a(m),b(n)]=[a,b](m+n)+Tr(ab)mc\delta_{m+n,0}
\end{array}
\end{equation}
for all $m,n\in\mathbb Z, a,b\in sl_2(\mathbb{C})$. The Chevalley basis of $sl_2$ consists of $X,Y,H$:
$$H=\left(\begin{array}{cc}
     1 & 0\\
     0 & -1
   \end{array}\right),
   X=\left(\begin{array}{cc}
     0 & 1\\
     0 & 0
   \end{array}\right),
   Y=\left(\begin{array}{cc}
     0 & 0\\
     1 & 0
   \end{array}\right)
$$
with brackets
\begin{equation}
[H,X]=2X,[H,Y]=-2Y,[X,Y]=H
\end{equation}

 Besides the presentation of ${A^{(1)}_1}$ with the basis $\{H(p),X(m),Y(n),c\mid p,m,n\in \mathbb{Z}\}$, satisfying the commutation relations (\ref{eq:2.1}), there is also the Kac-Moody definition by the Chevalley generators $\big\{h_i,e_j,f_k|i,j,k\in\{0,1\}\big\}$, subject to the conditions
\begin{equation}
[h_i,e_j]=A_{ij}e_j,[h_i,f_j]=-A_{ij}f_j,[e_i,f_j]=\delta_{ij}h_j,
\end{equation}
where $A=(A_{ij})=\left(\begin{array}{cc}
     2 & -2\\
     -2 & 2
   \end{array}\right)
$ is the generalized Cartan matrix.

The two equivalent descriptions of ${A^{(1)}_1}$ are related under the following correspondence
$$
\begin{array}{cc}
&e_0\leftrightarrow Y(1),f_0\leftrightarrow X(-1),h_0=-H(0)+c,\\
&e_1\leftrightarrow X(0),f_1\leftrightarrow Y(0),h_1=H(0)
\end{array}
$$

Recall that the weight space $V_\mu =\{v\in V\mid h\cdot v=\mu(h)v,\forall h\in \eta\}$, where $\eta=(\oplus_{n\in \mathbb{Z}}H(n))\oplus \mathbb{C}c\oplus \mathbb{C}d$ is the Cartan  subalgebra of $\widehat{sl}_2(\mathbb{C})$. A highest weight module $V(\lambda)=\oplus_{\mu\leq \lambda} V_\mu$, or the highest weight representation , is the space generated by a highest weight vector $v_{\lambda}$ of weight $\lambda$ such that $e_iv_{\lambda}=0, h_iv_{\lambda}=\lambda(h_i)v_{\lambda}$.
The central element $c$ acts on $V(\lambda)$ as a scalar $k$, which will be called the level of the
module.

The elements of Heisenberg subalgebra ${\mathfrak h}'=(\oplus_{n\neq0}H(n))\oplus \mathbb{C}c$ of $A^{(1)}_1$
obey the following relations which are special cases of Eq. (\ref{eq:2.1}):
\begin{equation}
[H(m),H(n)]=2mc\delta_{m+n,0}
\end{equation}

Given level $-k$, any negative integer. Let $S({\mathfrak h'}^-)$ the space of symmetric polynomials generated by elements in ${\mathfrak h'}^-=\oplus_{n<0}H(n)$. Then there is a canonical representation of the Heisenberg algebra ${\mathfrak h}'$ on $S({\mathfrak h'}^-)$ via the actions in accordance with Eq. (\ref{eq:2.1}):
\begin{equation}
\begin{array}{lll} &c\cdot 1=k,H(0)\cdot v=0,\\&H(m)\cdot v=H(m)v,m<0\\&H(m)\cdot v=-2mk\partial_{H(m)}(v),m>0
\end{array}
\end{equation}

In fact, this can be verified pretty straightforward, one needs only to observe that $[H(m),H(n)]\cdot v=-2mk\delta_{{_m+n,0}}v$ is valid under bracket relations.

We denote by $a(z)=\sum_{m\in \mathbb{Z}}a(m)z^{-m}$ the power formal series. Here and later, $z,w$ mean any formal variables. In this form,
${A^{(1)}_1}$  is usually called a current algebra. To write commutation relations in formal series, we need to introduce the formal $\delta$-function $\delta(\frac{w}{z})=\sum_{m\in \mathbb{Z}}(\frac{w}{z})^n$ , which possesses the fundamental property: for any $f(w,z)\in End(V)[[w,w^{-1},z,z^{-1}]]$ such that
$$lim_{w\rightarrow z}f(w,z)=f(z,z), \quad f(w,z)\delta(\frac{w}{z})=f(z,z)\delta(\frac{w}{z})$$
exists. More information on delta functions can be found in \cite{s15}.

The commutation relations of the affine Lie algebra can now be given as follows.
\begin{equation}\mathrm{}\label{eq:2.6}
\begin{array}{lll} &[H(z),X(w)]=2X(w)\delta(\frac{w}{z})\\&[H(z),Y(w)]=-2X(w)\delta(\frac{w}{z})
\\&[X(z),Y(w)]=H(w)\delta(\frac{w}{z})-kw\partial_w\delta(\frac{w}{z})
\end{array}
\end{equation}\\

\subsection{Parafermions}
 We now discuss the parafermion theory \cite{s14}. Let $\Phi$ be the root system of the simple Lie algebra $\mathfrak g$ and let $M$ (or $M \, mod \, kM_L$) denote the root lattice spanned by $\Phi$, where $-k$ is identified with the level in the corresponding affine Lie algebra $\hat{g}$ and $M_L$ is the long root sublattice. Let
 $E_{\alpha}$ be the root vector of $\mathfrak g$, and we normalize the Chevalley basis of $\mathfrak g$
 via $[E_{\alpha}, E_{\beta}]=\epsilon_{\alpha\beta}E_{\alpha+\beta}$ if $\alpha+\beta\in\Phi$. It is well-known that $\epsilon_{\alpha\beta}\in\mathbb Z$.

 General parafermion is defined for elements of $M$, but we will focus on parafermionic fields $\boldsymbol{\psi}_{\alpha}(z), \boldsymbol{\psi}_{\beta}(w)$ for
 roots $\alpha, \beta\in\Phi$ \cite{s13}.
 For two such parafermions the radial ordered product is defined as a multivalued function owning to the mutually semilocal property between them (cf. \cite{s14}). Instead of (anti-)commtativity the key relation is
\begin{equation}\label{eq:2.7}
R(\boldsymbol{\psi}_\alpha(z)\boldsymbol{\psi}_\beta(w))=(-1)^{\frac{\alpha\beta}{-k}}
R(\boldsymbol{\psi}_\beta(w)\boldsymbol{\psi}_\alpha(z)).
\end{equation}

For simplicity we will drop the symbol $R$. For $\alpha,\beta \in \Phi$ the operator product expansion for two parafermions can be formulated as (cf. \cite{s9})
\begin{equation}\label{eq:2.8}
\boldsymbol{\psi}_\alpha(z)\boldsymbol{\psi}_\beta(w)(z-w)^{-\frac{\alpha\beta}{k}}
=N\big(\boldsymbol{\psi}_\alpha(z)\boldsymbol{\psi}_\beta(w)\big) +\frac{\varepsilon_{\alpha,\beta}}{z-w}{\boldsymbol{\psi}_{\alpha+\beta}(w)}+
\frac{\delta_{\alpha+\beta,0}I_{\psi_{\alpha}(z)\psi_{\beta}(w)}}{(z-w)^2}
\end{equation}
in which $N\big(\boldsymbol{\psi}_\alpha(z)\boldsymbol{\psi}_\beta(w)\big)$ can be seen as infinitesimal of higher order in $z-w$, namely inside $\mathcal{O}(z-w)$. Note that the regular part of the expression
in parentheses satisfy
\[N\big(\boldsymbol{\psi}_\alpha(z)\boldsymbol{\psi}_\beta(w)\big)=N\big(\boldsymbol{\psi}_\beta(w)
\boldsymbol{\psi}_\alpha(z)\big),
\]
and
$$
\varepsilon_{\alpha,\beta}=\left\{\begin{array}{ll}\epsilon_{\alpha\beta}/{\sqrt{-k}}, & \textrm{if $\alpha+\beta\in\Phi$}\\
$0$, & \textrm{otherwise}\end{array}\right. ,
$$
where $I_{\psi_\alpha(z)\psi_\beta(w)}$ are some constants to be fixed later.

According to the parafermion theory, the conformal dimension of $\boldsymbol{\psi}_l(z), \, l\in\Phi$ is defined by $\Delta_l=\frac{l^2}{2k}+n(l)$ \cite{s13}, where $n(l)$ is the minimal number of roots $\alpha_i$ in $\Phi$ by which $l$ can be composed,
$\alpha=\sum_{i=1}^{n(l)}\alpha_i$. Note that Eq. (\ref{eq:2.8}) can be equivalently written as
\begin{equation}\label{eq:2.9}
\boldsymbol{\psi}_\alpha(z)\boldsymbol{\psi}_\beta(w)=(z-w)^{\Delta_{\alpha+\beta}-
\Delta_\alpha-\Delta_\beta}[\delta_{\alpha+\beta,0}I_{\boldsymbol{\psi}_{\alpha}(z)\boldsymbol{\psi}_{\beta}(w)}
+\varepsilon_{\alpha,\beta}\boldsymbol{\psi}_{\alpha+\beta}(w)+\ldots].
\end{equation}
In this case $\Delta_{\alpha}=\Delta_{-\alpha}=1$ and $\Delta_0=0$, therefore Eq. (\ref{eq:2.9}) is simply
\begin{align*}
&\boldsymbol{\psi}_\alpha(z)\boldsymbol{\psi}_\beta(w)(z-w)^{-\frac{\alpha\beta}{k}}\\
&\qquad=(z-w)^{n(\alpha+\beta)-2}[
\delta_{\alpha+\beta,0}I_{\psi_{\alpha}(z)\psi_{\beta}(w)}+
\varepsilon_{\alpha,\beta}\boldsymbol{\psi}_{\alpha+\beta}+\mathcal{O}(z-w)].
\end{align*} 

For
$\boldsymbol{\psi}_{\pm\alpha}(z)$ associated to Lie algebra $sl_2(\mathbb{C})$, we find that $I_{\psi_{\alpha}(z)\psi_{\beta}(w)}=-kzw$ for $\alpha=-\beta$, and it is $1$ otherwise. We define the normal ordered product
\begin{equation}\label{eq:2.10}
\boldsymbol{:}\boldsymbol{\psi}_\alpha(z)\boldsymbol{\psi}_\beta(w)\boldsymbol{:}(z-w)^{-\frac{\alpha\beta}{k}}
=N\big(\boldsymbol{\psi}_\alpha(z)\boldsymbol{\psi}_\beta(w)\big).
\end{equation}

We define the contraction function by
\begin{equation}\label{eq:2.11}
\underbrace{\boldsymbol{\psi}_\alpha(z)\boldsymbol{\psi}_\beta(w)}(z-w)^{-\frac{\alpha\beta}{k}}= \boldsymbol{\psi}_\alpha(z)\boldsymbol{\psi}_\beta(w)(z-w)^{-\frac{\alpha\beta}{k}}-\boldsymbol{:}\boldsymbol{\psi}_\alpha(z)\boldsymbol{\psi}_\beta(w)\boldsymbol{:}(z-w)^{-\frac{\alpha\beta}{k}}  \end{equation}

 Then Eq. (\ref{eq:2.8}) can be written by
\begin{equation}\label{eq:2.12} \boldsymbol{\psi}_\alpha(z)\boldsymbol{\psi}_\beta(w)(z-w)^{-\frac{\alpha\beta}{k}}=\left\{\begin{array}{ll}
\boldsymbol{:}\boldsymbol{\psi}_\alpha(z)\boldsymbol{\psi}_\alpha(w)\boldsymbol{:}(z-w)^{-\frac{2}{k}}, & if  \alpha=\beta\\[0.2cm]
\boldsymbol{:}\boldsymbol{\psi}_\alpha(z)\boldsymbol{\psi}_{-\alpha}(w)\boldsymbol{:}(z-w)^{\frac{2}{k}}+\frac{-kzw}{(z-w)^2},& if \alpha=-\beta\end{array}\right.
\end{equation}\\

Now for the affine Lie algebra $\widehat{sl}_2$ we define the $Z$-algebra operators $A_{\alpha}(z)$, $A_{-\alpha}^{\ast}(z)$ for $\boldsymbol{\psi}_\alpha(z),\boldsymbol{\psi}_{-\alpha}(z)$. Using Eq. (\ref{eq:2.12}), we get the following
equations:
\begin{equation}\label{eq:2.13}
\begin{array}{lll} &\underbrace{A_{\alpha}(z)A_{\alpha}(w)}(z-w)^{-\frac{2}{k}}=
\underbrace{A_{-\alpha}^{\ast}(z)A_{-\alpha}^{\ast}(w)}(z-w)^{-\frac{2}{k}}=0, \\[0.1cm]
&\underbrace{A_{\alpha}(z)A_{-\alpha}^{\ast}(w)}(z-w)^{\frac{2}{k}}=\frac{-kzw}{(z-w)^2},\\[0.1cm]
&\underbrace{A_{-\alpha}^{\ast}(w)A_{\alpha}(z)}(w-z)^{\frac{2}{k}}=\frac{-kzw}{(w-z)^2}.
\end{array}
\end{equation}

The operator $A_\alpha$ (or $A^*_{\alpha}$) acts on the field operator $\Phi_{\lambda,\bar{\lambda}}(w,\bar{w})$ with charge $(\lambda,\bar{\lambda})$ (cf. \cite{s9,s11,s12,s13}) as follows.
\begin{equation}\label{eq:2.14}
A_\alpha(z)\Phi_{\lambda,\bar{\lambda}}(w,\bar{w})=\sum_{m=\infty}^
\infty(z-w)^{-m-1+\frac{\alpha\lambda}{k}}A_{m}^{\alpha,\lambda}\Phi_{\lambda,\bar{\lambda}}(w,\bar{w}),
\end{equation}
then the component operator $A_{m}^{\alpha,\lambda}$ acts on $\Phi_{\lambda,\bar{\lambda}}$ via
\[
A_{m}^{\alpha,\lambda}\Phi_{\lambda,\bar{\lambda}}(w,\bar{w})=
\int_w{\frac{dz}{2\pi{i}}(z-w)^{m-\frac{\alpha\lambda}{2k}}\boldsymbol{A}_\alpha(z)
\Phi_{\lambda,\bar{\lambda}}(w,\bar{w})}. \]

We are interested only in parafermions $\boldsymbol{\psi}_\alpha$, carrying charge $(\alpha,0)$ with $\Phi_{\alpha,0}(w,\bar{w})$ \cite{s13}:
\begin{equation}\label{eq:2.15}
A_{\alpha}(z)\Phi_{\alpha,0}(w,\bar{w})=
\sum_{m=\infty}^{\infty}(z-w)^{-m-1+\frac{2}{k}}{A_{m}^{\alpha,\alpha}\Phi_{\alpha,0}(w,\bar{w})}.
\end{equation}

\subsection{Action of the group algebra}
The group algebra $\mathbb{C}
(\mathbb{Z}\alpha)$ is the associative algebra generated by $e^{n\alpha}$ ($n\in\mathbb Z$) under the multiplication
\begin{equation}\label{eq:3.1}
e^0=1, \qquad e^{m\alpha}\cdot e^{n\alpha}=e^{(m+n)\alpha}.
\end{equation}
where $m, n\in\mathbb{Z}$. The group algebra $\mathbb{C}
(\mathbb{Z}\alpha)$ acts on itself via multiplication, and we also introduce the operator $h(0)$ ($h\in \mathfrak h$) which acts on $\mathbb{C}
(\mathbb{Z}\alpha)$ by
$$
\begin{array}{ll}
h(0):&\mathbb{C}(\mathbb{Z}\alpha)\rightarrow\mathbb{C}(\mathbb{Z}\alpha)\\[0.2cm]
&\mspace{33mu}e^\alpha\mapsto\langle h,\alpha \rangle e^\alpha
\end{array}
$$
so we get $[h(0),e^\alpha]=\langle h,\alpha\rangle e^\alpha$.
Using the operator $h(0)$ we naturally define the operator $z^h\in(End\mathbb{C}(\mathbb{Z}\alpha))\{z\}$(can be seen as $ z^{h(0)})$ for $h\in\mathbb{Z}\alpha$ by
\begin{equation}\label{eq:3.2}
z^h\cdot e^{\alpha}=z^{\langle h,\alpha\rangle}e^\alpha.
\end{equation}
Then we get
\begin{equation}\label{eq:3.3}
\begin{array}{cc}
[\alpha(0),z^\beta]=0,\\[0.2cm]
z^{\alpha}e^{\beta}=z^{\langle\alpha,\beta\rangle}e^{\beta}z^{\alpha}=e^{\beta}z^{\alpha+\langle\alpha,\beta\rangle}.
\end{array}
\end{equation}


\section{Construction of the Parafermion Representations of ${A^{(1)}_1}$ and $Z$-algebra}
\label{sec:3}
\subsection{Action of Heisenberg subalgebra}
We define the following exponential operators on the space $S({\mathfrak h'}^-)$ and their properties are given in Proposition \ref{prop:3.1}.
$$
\begin{array}{cc}
&E_+^{\pm}(z)=exp(\mp\sum\limits_{n>0}\frac{H(-n)}{kn}z^n)\\
&E_-^{\pm}(z)=exp(\pm\sum\limits_{n>0}\frac{H(n)}{kn}z^{-n})
\end{array}
$$
\\

\begin{proposition}\label{prop:3.1} On the space $S({\mathfrak h'}^-)$ we have
\begin{equation}\label{eq:3.1}
\begin{array}{lll}
&E^+_{\pm}(z)E^-_{\pm}(z)=E^-_{\pm}(z)E^+_{\pm}(z)=1\\[0.2cm]
&E^+_+(z)E^{\mp}_+(z)=E^{\mp}_+(z)E^+_+(z)\\[0.2cm]
&E^-_+(z)E^{\mp}_{+}(z)=E^{\mp}_{+}(z)E^-_+(z)\\[0.2cm]
&\partial_z(E_+^{\pm}(z)E_-^{\pm}(z))=\mp E_-^{\pm}(z)E_+^{\pm}(z)\sum\limits_{n\neq0}\frac{H(n)}{k}z^{-n-1}\end{array}
\end{equation}

\begin{equation}\label{eq:3.2}
\begin{array}{ll}
&E^+_{\pm}(z)E^+_{\pm}(w)=E^+_{\pm}(w)E^+_{\pm}(z)\mspace{155mu}\\[0.2cm]
&E^-_{\pm}(z)E^-_{\pm}(w)=E^-_{\pm}(w)E^-_{\pm}(z)\
\end{array}  
\end{equation}

\begin{equation}
\begin{array}{lll}\label{eq:3.3}
&E^+_{\pm}(z)E^-_{\pm}(w)=E^-_{\pm}(w)E^+_{\pm}(z)\mspace{155mu}\\[0.2cm]
&E^{\pm}_+(z)E^{\mp}_-(w)=E^{\mp}_-(w)E^{\pm}_+(z)(1-\frac{z}{w})^{-\frac{2}{k}}\\[0.2cm]
&E^{\pm}_+(z)E^{\pm}_-(w)=E^{\pm}_-(w)E^{\pm}_+(z)(1-\frac{z}{w})^{\frac{2}{k}}
\end{array} 
\end{equation}
\end{proposition}
\begin{proof} These identities are proved by the Campbell-Hausdorf-Witt theorem. The commutativity relations are easy
consequence of the fact that $H(m)$ and $H(n)$ commute if $m\neq -n$. For the other identities we compute that
$$
\partial_z(E^-_+(z)E^-_-(z))=\partial_z(exp(-\textstyle\sum\limits_{n\neq0}\frac{H(n)}
{kn}z^{-n}))=E^-_+(z)E^-_-(z)\sum\limits_{n\neq0}\frac{H(n)}{k}z^{-n-1}.
$$

For the last two relations in Eq. (\ref{eq:3.3}), we use the identity $e^{x_1}e^{x_2}=e^{x_2}e^{x_1}e^{[x_1,x_2]}$ if $x_1$, $x_2$ commute
with $[x_1, x_2]$:
$$
\begin{array}{lll}
E^+_+(z)E^-_-(w)&=E^-_+(w)E^+_+(z)exp([\sum\limits_{m>0}\frac{-H(-m)}{km}z^m,-\sum\limits_{n>0}\frac{H(n)}{kn}w^{-n}])\\
&=E^-_+(w)E^+_+(z)exp(-\sum\limits_{m,n>0}\frac{2mc\delta_{m-n,0}}{k^2mn}z^mw^{-n})\\
&=E^-_+(w)E^+_+(z)(1-\frac{z}{w})^{-\frac{2}{k}}.
\end{array}
$$
\qed
\end{proof}

\subsection{The realization}
Let $V=S({\mathfrak h'}^-)\otimes <\Phi_{\alpha,0}(\omega,\overline{\omega})>\otimes \mathbb{C}(\mathbb{Z}\alpha)$,
we define the map $\pi :\widehat{sl}_2(\mathbb{C})\rightarrow End(V)\{z\}$ as follows:
$$
\begin{array}{llll}
&X(z)\longmapsto E^+_+(z)E^+_-(z)\otimes A_\alpha(z)e^\alpha z^{-\frac{\alpha}{k}}\\
&Y(z)\longmapsto E^-_+(z)E^-_-(z)\otimes A_{-\alpha}^{\ast}(z)e^{-\alpha} z^{\frac{\alpha}{k}}\\
&H(z)\longmapsto H(z)\otimes 1\\
&c\longmapsto -k\\
&d\longmapsto deg.
\end{array}
$$
\begin{theorem}\label{th:3.2}
$(\pi,V)$ defines a representation of ${A^{(1)}_1}$.
\end{theorem}
\begin{proof} For convenience, we just check the relations in Eq. (\ref{eq:2.6}).
\begin{eqnarray*}X(z)X(w)&\mapsto &\mspace{-8mu}E^+_+(z)E^+_-(z)E^+_+(w)E^+_-(w)\otimes A_\alpha(z)e^\alpha z^{-\frac{\alpha}{k}} A_\alpha(w)e^\alpha w^{-\frac{\alpha}{k}}\\
&=&\mspace{-8mu}E^+_+(z)E^+_+(w)E^+_-(z)E^+_-(w)(1-\frac{w}{z})^{-\frac{2}{k}}\otimes A_{\alpha}(z)A_\alpha(w) e^{2\alpha}z^{-\frac{\alpha}{k}-\frac{2}{k}}w^{-\frac{\alpha}{k}}\\
&=&\mspace{-8mu}E^+_+(z)E^+_+(w)E^+_-(z)E^+_-(w)\otimes\boldsymbol{:}A_{\alpha}(z)A_\alpha(w)\boldsymbol{:}(z-w)^{-\frac{2}{k}}e^{2\alpha}(zw)^{-\frac{\alpha}{k}}\\
&&\mspace{-8mu}+E^+_+(z)E^+_+(w)E^+_-(z)E^+_-(w)\otimes\underbrace{A_{\alpha}(z)A_\alpha(w)}(z-w)^{-\frac{2}{k}}e^{2\alpha}(zw)^{-\frac{\alpha}{k}}\\
&=&\mspace{-8mu}E^+_+(z)E^+_+(w)E^+_-(z)E^+_-(w)\otimes\boldsymbol{:}A_{\alpha}(z)A_\alpha(w)\boldsymbol{:}(z-w)^{-\frac{2}{k}}e^{2\alpha}(zw)^{-\frac{\alpha}{k}}
\end{eqnarray*}

Hence,\begin{align*}
&\lefteqn{[X(z),X(w)]\;= X(z)X(w)-X(w)X(z)}\\
&=E^+_+(z)E^+_+(w)E^+_-(z)E^+_-(w)\otimes\big(\boldsymbol{:}
A_{\alpha}(z)A_\alpha(w)\boldsymbol{:}(z-w)^{-\frac{2}{k}}\\
&\qquad\qquad-\boldsymbol{:}A_{\alpha}(w)A_\alpha(z)\boldsymbol{:}(w-z)^{-\frac{2}{k}}\big)
e^{2\alpha}(zw)^{-\frac{\alpha}{k}}\\
&=0
\end{align*}

By similar method we get $[Y(z)Y(w)]=0$. Next notice that
\begin{eqnarray*}
\lefteqn{X(z)Y(w)\;\mapsto E^+_+(z)E^+_-(z)E^-_+(w)E^-_-(w)\otimes A_\alpha(z)e^\alpha z^{-\frac{\alpha}{k}}A_{-\alpha}^{\ast}(w)e^{-\alpha} w^{\frac{\alpha}{k}}}\\
&&=E^+_+(z)E^-_+(w)E^+_-(z)E^-_-(w)(1-\frac{w}{z})^{\frac{2}{k}}\otimes\big(\boldsymbol{:}A_{\alpha}(z)A_{-\alpha}^{\ast}(w)\boldsymbol{:}+\underbrace{A_{\alpha}(z)A_{-\alpha}^{\ast}(w)}\big)e^{\alpha-\alpha}z^{-\frac{\alpha}{k}+\frac{2}{k}}w^{\frac{\alpha}{k}}\\
&&=E^+_+(z)E^-_+(w)E^+_-(z)E^-_-(w)\otimes\boldsymbol{:}A_{\alpha}(z)A_{-\alpha}^{\ast}(w)\boldsymbol{:}(z-w)^{\frac{2}{k}}z^{-\frac{\alpha}{k}}w^{\frac{\alpha}{k}}\\
&&\mspace{18mu}+E^+_+(z)E^-_+(w)E^+_-(z)E^-_-(w)\otimes\frac{-kzw}{(z-w)^2}z^{-\frac{\alpha}{k}}w^{\frac{\alpha}{k}}
\end{eqnarray*}

Thus we get the computation as expected
\begin{eqnarray*}
\lefteqn{[X(z),Y(w)]=X(z)Y(w)-Y(w)X(z)}\\
&&\mspace{-8mu}=E^+_+(z)E^-_+(w)E^+_-(z)E^-_-(w)\otimes\big(\boldsymbol{:}A_{\alpha}(z)A_{-\alpha}^{\ast}(w)\boldsymbol{:}(z-w)^{\frac{2}{k}}-\boldsymbol{:}A_{-\alpha}^{\ast}(w)A_{\alpha}(z)\boldsymbol{:}(w-z)^{\frac{2}{k}}\big)z^{-\frac{\alpha}{k}}w^{\frac{\alpha}{k}}\\
&&\mspace{-8mu}+E^+_+(z)E^-_+(w)E^+_-(z)E^-_-(w)\otimes\big(\frac{-kzw}{(z-w)^2}-\frac{-kzw}{(w-z)^2}\big)z^{-\frac{\alpha}{k}}w^{\frac{\alpha}{k}}\\
&&\mspace{-8mu}=-kE^+_+(z)E^-_+(w)E^+_-(z)E^-_-(w)z^{-\frac{\alpha}{k}}w^{\frac{\alpha}{k}}w\partial_{_w}{\delta(\frac{w}{z})}\\
&&\mspace{-8mu}=-kw\partial_{_w}\big(E^+_+(z)E^-_+(w)E^+_-(z)E^-_-(w)z^{-\frac{\alpha}{k}}w^{\frac{\alpha}{k}}\delta(\frac{w}{z})\big)+kw\partial_{_w}(E^+_+(z)E^-_+(w)E^+_-(z)E^-_-(w))z^{-\frac{\alpha}{k}}w^{\frac{\alpha}{k}}\delta(\frac{w}{z})\\
&&\mspace{-8mu}+kwE^+_+(z)E^-_+(w)E^+_-(z)E^-_-(w)\partial_{_w}(z^{-\frac{\alpha}{k}}w^{\frac{\alpha}{k}})\delta(\frac{w}{z})\\
&&\mspace{-8mu}=-kw\partial_{_w}{\delta(\frac{w}{z})}+kw\textstyle\sum\limits_{n\neq0}\frac{H(n)}{k}w^{n-1}w^{\frac{\alpha}{k}}z^{-\frac{\alpha}{k}}\delta(\frac{w}{z})+\alpha w^{\frac{\alpha}{k}}z^{-\frac{\alpha}{k}}\delta(\frac{w}{z})\\
&&\mspace{-8mu}=-kw\partial_{_w}{\delta(\frac{w}{z})}+\textstyle\sum\limits_{n\neq0}{H(n)}w^{-n}\delta(\frac{w}{z})+H(0)\delta(\frac{w}{z})\\
&&\mspace{-8mu}=H(w)\delta(\frac{w}{z})-kw\partial_{_w}{\delta(\frac{w}{z})}.
\end{eqnarray*}

It is easy to compute that
\begin{eqnarray*}
[H(z),E^-_+(w)]&=&\mspace{-8mu}\textstyle\sum\limits_{m\in\mathbb{Z}}[H(m),e^{\sum\limits_{n>0}\frac{H(-n)}{kn}w^n}]z^{-m}=
E^-_+(w)\sum\limits_{m\in\mathbb{Z}\atop n>0}{\frac{[H(m),H(-n)]}{kn}}z^{-m}w^n\\
&=&E^-_+(w)\textstyle\sum\limits_{m\in\mathbb{Z}\atop n>0}
{\frac{2mc\delta_{m-n,0}}{kn}z^{-m}w^n=E^-_+(w)\sum\limits_{n>0}-2z^{-n}w^n}.
\end{eqnarray*}

A similar calculation for $[H(z),E^+_+(w)]$,$[H(z),E^-_-(w)]$,$[H(z),E^+_-(w)]$ yields
\begin{eqnarray*}
\lefteqn{[H(z),X(w)]=[H(z),E^+_+(w)]E^+_-(w)\otimes A_\alpha(w)e^\alpha w^{-\frac{\alpha}{k}}}\\
&&\mspace{18mu}+E^+_+(w)[H(z),E^+_-(w)]\otimes A_\alpha(w)e^\alpha w^{-\frac{\alpha}{k}}+E^+_+(w)E^+_-(w)\otimes A_\alpha(w)[H(0),e^\alpha]w^{-\frac{\alpha}{k}}\\
&&=2E^+_+(w)E^+_-(w)\big(\textstyle\sum\limits_{n>0}z^{-n}w^n+\textstyle\sum\limits_{n<0}z^{-n}w^n+1\big)\otimes A_\alpha(w)e^\alpha w^{-\frac{\alpha}{k}}\\
&&=2E^+_+(w)E^+_-(w)\otimes A_\alpha(w)e^\alpha w^{-\frac{\alpha}{k}}=2X(w)\delta(\frac{w}{z}).
\end{eqnarray*}

It is immediate that $[H(z),Y(w)]=-2Y(w)\delta(\frac{w}{z})$. \qed
\end{proof}

The action of $c$ shows that the representation of ${A^{(1)}_1}$ just obtained has the level $-k$.

\subsection{The $Z$-algebra}
Furthermore we can get the representation of $Z$-algebra as in \cite{s6}. We remark that this $Z$-algebra is fundamentally different from the $Z$-algebra in \cite{s14}. Taken the same definition of formal power series $Z^{\pm}(z), x(\phi_1,z),x(\phi_2,z)$:
 $$Z^{+}(z)=Z(\alpha,z)=E^-_+(z)X(z)E^-_-(z), Z^{-}(z)=Z(-\alpha,z)=E^+_+(z)Y(z)E^+_-(z),
$$
and the generalized commutator brackets
\newcommand{\dlbrack}{[\hspace{-1.5pt}[}
\newcommand{\drbrack}{]\hspace{-1.5pt}]}
$$\dlbrack x(\phi_1,z),x(\phi_2,z)\drbrack=x(\phi_1,z)x(\phi_2,z)(1-\frac{w}{z})^{\frac{(\phi_1,\phi_2)}{c}}-x(\phi_2,z)x(\phi_1,z)(1-\frac{w}{z})^{\frac{(\phi_1,\phi_2)}{c}},$$for
$\phi_1,\phi_2=\pm\alpha$,we can check that the lemmas given in paper \cite{s6} still hold. We state them here in Lemma
\ref{lem:3.3} and Lemma \ref{lem:3.4}.
\\
\\
\begin{lemma}\label{lem:3.3}
Let $Z$-operators $Z(z)=Z^{+}(z),Z^{-}(z)$, we have that
\begin{equation}
[E^\pm_+(z),Z(w)]=0, \qquad[E^\pm_-(z),Z(w)]=0
\end{equation}
\end{lemma}
{\bf Proof. } For $n\neq0$, simple calculation yields
$$[H(n),X(w)]=\textstyle\sum\limits_{m\in\mathbb{Z}}[H(n),X(m)]w^{-m}=\sum\limits_{m\in\mathbb{Z}}2X(m+n)w^{-m}=2X(w)w^{n}$$
and write $x_1=\partial_{s}(e^{sx_1})|_{s=0}$, the following equations follow from $[x_1,e^{x_2}]=e^{x_2}[x_1,x_2]$
\begin{eqnarray*}
[H(n),E^{-}_-(w)]&=&\mspace{-8mu}E^{-}_-(w)[H(n),\textstyle\sum\limits_{m>0}\frac{H(m)}{km}w^{-m}]\\
&=&\mspace{-8mu}E^{-}_-(w)\textstyle\sum\limits_{m>0}\frac{2nc\delta_{m+n,0}}{km}w^{-m}=-2E^-_-(w)w^{n}\delta_{_{-m,n<0}},
\end{eqnarray*}
\begin{eqnarray*}
[H(n),E^{-}_+(w)]=-2E^-_+(w)w^{n}\delta_{_{m,n>0}}.
\end{eqnarray*}

Note that\begin{eqnarray*}
\lefteqn{[H(n),Z^+(w)]\;= [H(n),E^-_+(w)X(w)E^-_-(w)]}\\
&&=[H(n),E^-_+(w)]X(w)E^-_-(w)+E^-_+(w)[H(n),X(w)]E^-_-(w)+E^-_+(w)X(w)[H(n),E^-_-(w)]\\
&&=\big(-2w^{n}\delta_{_{m,n>0}}+2w^{n}+2w^{n}\delta_{_{-m,n<0}}\big)E^-_+(w)X(w)E^-_-(w)=0.
\end{eqnarray*}

Similiarly, $[H(n),Z^-(w)]=[H(-n),Z^\pm(w)]=0$, so
$$[E^+_-(z),Z^+(w)]=E^+_-(z)\big{[}\sum\limits_{n>0}\frac{H(n)}{kn}z^{-n},Z^+(w)\big{]}=0$$

The calculation for the other brackets is similar.
\\
\\
\begin{lemma}\label{lem:3.4} One has
\begin{equation}
\begin{array}{ll}
&\dlbrack Z^{\pm}(z),Z^{\pm}(w)\drbrack=0\\
&\dlbrack Z^{+}(z),Z^{-}(w)\drbrack=H(0)\delta(\frac{w}{z})-kw\partial_w\delta(\frac{w}{z})
\end{array}
\end{equation}
\end{lemma}
\begin{proof} We calculate the product using Proposition \ref{prop:3.1} together with Lemma \ref{lem:3.3}
\begin{eqnarray*}
Z^-(z)Z^-(w)&=&\mspace{-8mu}Z^-(z)E^+_+(w)Y(w)E^+_-(w)E^+_+(w)Z^{-}(z)Y(w)E^+_-(w)\\
&=&\mspace{-8mu}E^+_+(w)E^+_+(z)Y(z)\big{(}E^-_+(w)E^+_+(w)\big{)}E^+_-(z)Y(w)E^+_-(w)\\
&=&\mspace{-8mu}E^+_+(w)E^+_+(z)Y(z)E^-_+(w)E^+_-(z)E^+_+(w)(1-\frac{w}{z})^{\frac{2}{k}}Y(w)E^+_-(w)\\
&=&\mspace{-8mu}E^+_+(w)E^+_+(z)Y(z)E^-_+(w)\big{(}E^+_-(z)Z^-(w)\big{)}(1-\frac{w}{z})^{\frac{2}{k}}\\
&=&\mspace{-8mu}E^+_+(w)E^+_+(z)Y(z)E^-_+(w)E^+_+(w)Y(w)E^+_-(w)E^+_-(z)(1-\frac{w}{z})^{\frac{2}{k}}\\
&=&\mspace{-8mu}E^+_+(w)E^+_+(z)Y(z)Y(w)E^+_-(w)E^+_-(z)(1-\frac{w}{z})^{\frac{2}{k}}
\end{eqnarray*}

Consequently,
\begin{eqnarray*}
\dlbrack Z^-(z)Z^-(w)\drbrack&=&\mspace{-8mu}Z^-(z)Z^-(w)(1-\frac{w}{z})^{-\frac{2}{k}}-Z^-(w)Z^-(z)(1-\frac{z}{w})^{-\frac{2}{k}}\\
&=&\mspace{-8mu}E^+_+(z)E^+_+(w)[Y(z),Y(w)]E^+_-(z)E^+_-(w)=0
\end{eqnarray*}

The calculation for $\dlbrack Z^+(z),Z^+(w)\drbrack $ is similar.

Also
$$
\begin{array}{ll}
Z^+(z)Z^-(w)&=E^+_+(w)E^-_+(z)X(z)Y(w)E^+_-(w)E^-_-(z)(1-\frac{w}{z})^{-\frac{2}{k}}\\
Z^-(w)Z^+(z)&=E^-_+(z)E^+_+(w)Y(w)X(z)E^-_-(z)E^+_-(w)(1-\frac{z}{w})^{-\frac{2}{k}}\\
\end{array}
$$
\begin{eqnarray*}
\dlbrack \lefteqn{Z^+(z)Z^-(w)\drbrack\, =Z^+(z)Z^-(w)(1-\frac{w}{z})^{\frac{2}{k}}-Z^-(w)Z^+(z)(1-\frac{z}{w})^{\frac{2}{k}}}\\
&&\mspace{-20mu}=E^+_+(w)E^-_+(z)[X(z)Y(w)]E^+_-(w)E^-_-(z)\\
&&\mspace{-20mu}=E^-_+(z)E^+_+(w)H(w)\delta(\frac{w}{z})E^-_-(z)E^+_-(w)-E^-_+(z)E^+_+(w)kw\partial_w\delta(\frac{w}{z})E^-_-(z)E^+_-(w)\\
&&\mspace{-20mu}=H(w)\delta(\frac{w}{z})-kw\left(\partial_w\big{(}E^-_+(z)E^+_+(w)E^-_-(z)E^+_-(w)
\delta(\frac{w}{z})\big{)}\right.\\
&&\qquad\qquad\qquad\qquad-\left.\partial_w\big(E^-_+(z)E^+_+(w)E^-_-(z)E^+_-(w)\big)\delta(\frac{w}{z})\right)\\
&&\mspace{-20mu}=\textstyle\sum\limits_{m\in\mathbb
Z}{H(m)}w^{-m}\delta(\frac{w}{z})-kw\partial_w\delta(\frac{w}{z})+kw\textstyle\sum\limits_{m\neq0}\frac{H(m)}{k}w^{-m-1}\delta(\frac{w}{z})\\
&&\mspace{-20mu}=H(0)\delta(\frac{w}{z})-kw\partial_w\delta(\frac{w}{z})
\end{eqnarray*}
\qed
\end{proof}

For $A^{(1)}_1$-module $V=S({\mathfrak h'}^-)\otimes \langle\Phi_{\alpha,0}(\omega,\overline{\omega})\rangle\otimes \mathbb{C}(\mathbb{Z}\alpha)$ in Theorem \ref{th:3.2}, we define the vacuum space $\Omega(V)$ of $V$
by
$$\Omega (V)=\{v\in V,\eta{'}^{+}=\oplus_{n>0}H(n)\mid {\eta{'}}^{+}\cdot v =0\}.
$$
Observe that we can decompose  $V$ by $V=S({\mathfrak h'}^-)\otimes\Omega(V)$, then we get $\Omega(V)=\Phi_{\alpha,0}(\omega,\overline{\omega})\otimes \mathbb{C}(\mathbb{Z}\alpha)$ and furthermore
\begin{theorem}\label{th:3.5} The map
$\pi_{\Omega}: Z\rightarrow gl(\Omega(V))$
gives a representation of
$Z$-algebra on the vacuum space $\Omega (V)$ at level $-k$ via the action:
$$
\begin{array}{ll}
&Z^+(z)\longmapsto A_\alpha(z)e^\alpha z^{-\frac{\alpha}{k}}\\
&Z^-(z)\longmapsto A_{-\alpha}^{\ast}(z)e^{-\alpha} z^{\frac{\alpha}{k}}\\.
\end{array}$$
\end{theorem}
\begin{proof} Under the map $\pi$ we have
\begin{eqnarray*}
Z^+(z)Z^+(w)&\mapsto&\mspace{-8mu}A_\alpha(z)e^\alpha z^{-\frac{\alpha}{k}}A_\alpha(w)e^\alpha w^{-\frac{\alpha}{k}}\\
&=&\mspace{-8mu}A_\alpha(z)A_\alpha(w)e^{2\alpha}z^{-\frac{2}{k}}{(zw)}^{-\frac{\alpha}{k}}
\end{eqnarray*}

Therefore,
\begin{eqnarray*}
\dlbrack Z^+(z),Z^+(w)\drbrack&=&\mspace{-8mu}Z^+(z)Z^+(w)(1-\frac{w}{z})^{-\frac{2}{k}}-Z^+(w)Z^+(z)(1-\frac{z}{w})^{-\frac{2}{k}}\\
&\mapsto&\mspace{-8mu}\big(A_{\alpha}(z)A_\alpha(w)(z-w)^{-\frac{2}{k}}-A_{\alpha}(w)A_\alpha(z)(w-z)^{-\frac{2}{k}}\big)e^{2\alpha}{(zw)}^{-\frac{\alpha}{k}}\\
&=&\mspace{-8mu}\big(\boldsymbol{:}A_{\alpha}(z)A_\alpha(w)\boldsymbol{:}(z-w)^{-\frac{2}{k}}-\boldsymbol{:}A_{\alpha}(w)A_\alpha(z)\boldsymbol{:}(w-z)^{-\frac{2}{k}}\big)e^{2\alpha}{(zw)}^{-\frac{\alpha}{k}}\\
&&\mspace{-8mu}+\big(\underbrace{A_{\alpha}(z)A_\alpha(w)}(z-w)^{-\frac{2}{k}}-\underbrace{A_{\alpha}(w)A_\alpha(z)}(w-z)^{-\frac{2}{k}}\big)e^{2\alpha}{(zw)}^{-\frac{\alpha}{k}}=0
\end{eqnarray*}

Similar calculations produce$\dlbrack Z^-(z),Z^-(w)=0\drbrack$. Note that
\begin{eqnarray*}
Z^+(z)Z^-(w)&\mapsto&\mspace{-8mu}A_\alpha(z)e^\alpha z^{-\frac{\alpha}{k}}A_{-\alpha}^{\ast}(w)e^{-\alpha}w^{\frac{\alpha}{k}}\\
&=&\mspace{-8mu}A_\alpha(z)A_{-\alpha}^{\ast}(w)z^{\frac{2}{k}}z^{-\frac{\alpha}{k}}w^{\frac{\alpha}{k}}
\end{eqnarray*}
\begin{eqnarray*}
Z^-(w)Z^+(z)&\mapsto&\mspace{-8mu}A_{-\alpha}^{\ast}(w)A_{\alpha}(z)w^{\frac{2}{k}}z^{-\frac{\alpha}{k}}w^{\frac{\alpha}{k}}
\end{eqnarray*}
Then
\begin{eqnarray*}
\dlbrack Z^+(z),Z^-(w)\drbrack&=&\mspace{-8mu}Z^+(z)Z^-(w)(1-\frac{w}{z})^{\frac{2}{k}}-Z^-(w)Z^+(z)(1-\frac{z}{w})^{\frac{2}{k}}\\
&\mapsto&\mspace{-8mu}\big(A_{\alpha}(z)A_{-\alpha}^{\ast}(w)(z-w)^{\frac{2}{k}}-A_{-\alpha}^{\ast}(w)A_\alpha(z)(w-z)^{\frac{2}{k}}\big) z^{-\frac{\alpha}{k}}w^{\frac{\alpha}{k}}\\
&=&\mspace{-8mu}\big(\boldsymbol{:}A_{\alpha}(z)A_{-\alpha}^{\ast}(w)\boldsymbol{:}(z-w)^{\frac{2}{k}}-\boldsymbol{:}A_{-\alpha}^{\ast}(w)A_\alpha(z)\boldsymbol{:}(w-z)^{\frac{2}{k}}\big) z^{-\frac{\alpha}{k}}w^{\frac{\alpha}{k}}\\
&&\mspace{-8mu}+\big(\underbrace{A_{\alpha}(z)A_{-\alpha}^{\ast}}(z-w)^{\frac{2}{k}}-\underbrace{A_{-\alpha}^{\ast}(w)A_\alpha(z)}(w-z)^{\frac{2}{k}}\big)z^{-\frac{\alpha}{k}}w^{\frac{\alpha}{k}}\\
&=&\mspace{-8mu}\big(\frac{-kzw}{(z-w)^2}-\frac{-kzw}{(w-z)^2}\big)z^{-\frac{\alpha}{k}}w^{\frac{\alpha}{k}}=-kw\partial_w\delta(\frac{w}{z})z^{-\frac{\alpha}{k}}w^{\frac{\alpha}{k}}\\
&=&\mspace{-8mu}-kw\partial_w{\big(\delta(\frac{w}{z})z^{-\frac{\alpha}{k}}w^{\frac{\alpha}{k}}\big)}+kw\delta(\frac{w}{z})\partial_w(z^{\frac{\alpha}{k}}w^{-\frac{\alpha}{k}})\\
&=&\mspace{-8mu}-kw\partial_w{\delta(\frac{w}{z})}+az^{\frac{\alpha}{k}}w^{-\frac{\alpha}{k}}
\delta(\frac{w}{z})=H(0)\delta(\frac{w}{z})-kw\partial_w\delta(\frac{w}{z}),
\end{eqnarray*}
from which the theorem follows. \qed
\end{proof}

\end{document}